\documentclass{amsart}

\usepackage{graphics}

\usepackage[T1]{fontenc}    % use 8-bit T1 fonts

\usepackage{amsthm}
\usepackage{amsbsy,amsmath,amssymb,amscd,amsfonts}
\usepackage[pagebackref=true]{hyperref}
\usepackage[nameinlink,capitalize,noabbrev]{cleveref}
\usepackage{url}            % simple URL typesetting
\usepackage{booktabs}       % professional-quality tables
\usepackage{nicefrac}       % compact symbols for 1/2, etc.
\usepackage{microtype}      % microtypography

\usepackage{graphicx,float,latexsym,color}
\usepackage[font={small,it}]{caption}
\usepackage{subcaption}

\usepackage{makecell}

\usepackage[dvipsnames]{xcolor}

\newtheorem{theorem}{Theorem}
\newtheorem{observation}{Observation}
\newtheorem{proposition}{Proposition}
\newtheorem{conjecture}{Conjecture}
\newtheorem{remark}{Remark}
\newtheorem{corollary}{Corollary}

\hypersetup{
    pdftoolbar=true,        % show Acrobat’s toolbar?
    pdfmenubar=true,        % show Acrobat’s menu?
    pdffitwindow=false,     % window fit to page when opened
    pdfstartview={FitH},    % fits the width of 
    colorlinks=true,       % false: boxed links; true: colored links
    linkcolor=OliveGreen,          % color of internal links (change box color with linkbordercolor)
    citecolor=blue,        % color of links to bibliography
    filecolor=black,      % color of file links
    urlcolor=red           % color of external links
}

\usepackage{lineno}

\title[]{The Talented Mr. Inversive Triangle\\in the Elliptic Billiard}

\author{Dan Reznik$^*$}
\thanks{D. Reznik$^*$, Data Science Consulting, \texttt{dreznik@gmail.com}}
\author{Ronaldo Garcia} 
\thanks{R. Garcia, Federal Univ. of Goiás, \texttt{ragarcia@ufg.br}}
\author{Mark Helman}
\thanks{M. Helman, Rice University, \texttt{markhelman@hotmail.com}}

\begin{document}

\maketitle

\begin{abstract}
Inverting the vertices of elliptic billiard N-periodics with respect to a circle centered on one focus yields a new ``focus-inversive'' family inscribed in Pascal's Limaçon. The following are some of its surprising invariants: (i) perimeter, (ii) sum of cosines, and (iii) sum of distances from inversion center (the focus) to vertices. We prove these for the N=3 case, showing that this family (a) has a stationary Gergonne point, (b) is a 3-periodic family of a second, rigidly moving elliptic billiard, and (c) the loci of incenter, barycenter, circumcenter, orthocenter, nine-point center, and a great many other triangle centers are circles.

\end{abstract}

\section{Introduction}
%\label{sec:intro}
In \cite{reznik2020-invariants} we introduced the ``focus-inversive'' polygon, whose vertices are inversions of elliptic billiard N-periodic vertices with respect to a circle centered on one focus. Since N-periodics are inscribed in an ellipse, the focus-inversive family will be inscribed in Pascal's Limaçon \cite{ferreol2020-limacon}; see \cref{fig:inversive}, and is therefore non-Ponceletian.

A classic invariant of billiard N-periodics is perimeter \cite{sergei91}. Recently, the sum of cosines was proved as an interesting dependent\footnote{Any ``new'' invariants are dependent upon the two integrals of motion -- linear and angular momentum -- that render the elliptic billiard an integrable dynamical system \cite{kaloshin2018}.} invariant \cite{akopyan12,bialy2020-invariants}. Experimentally, for all $N$, the focus-inversive family conserves (i) the sum of distances from the focus to its vertices, (ii) perimeter, and (iii) the sum-of-cosines (except for $N=4$). The first two have been recently proved \cite{roitman2021-bicentric}.

In \cite[Thm 1]{schwartz2016-com} the loci of classic centroids of Poncelet polygons are studied. It is shown that while vertex and area centroids always sweep out conics, the locus of the perimeter centroid is in general non-conic. Interestingly, for all $N$, we find that the loci of the 3 aforementioned centroids are {\em circles} over the focus-inversive family, though a proof is still pending.

\begin{figure}
    \centering
    \includegraphics[width=.8\textwidth]{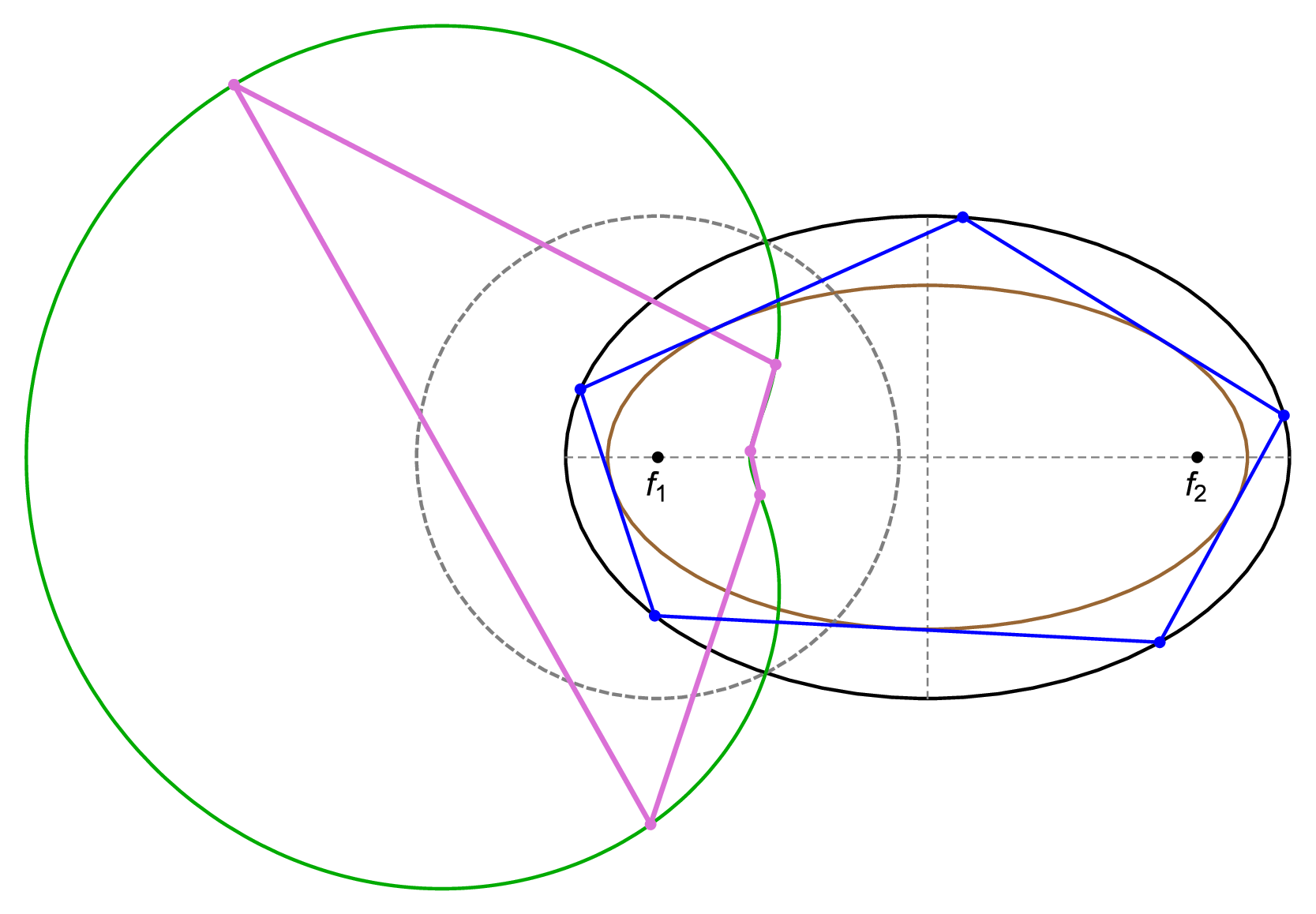}
      \caption{For a generic N-periodic (blue, N=5) with vertices $P_i$, the inversive polygon (pink) is inscribed in Pascal's Limaçon (green) and has vertices at inversions of the $P_i$ with respect to a circle (dashed black) centered on $f_1$. For all $N$, the focus-inversive perimeter, sum of cosines$^\ddagger$, and sum of focus-to-vertex distances are experimentally invariant \cite{reznik2020-invariants}; and a proof has recently appeared \cite{roitman2021-bicentric}. $^\ddagger$ Sum of focus-inversive cosines is variable for N=4 simple and a certain self-intersected N=6 \cite{garcia2020-self-intersected}. \href{https://youtu.be/wkstGKq5jOo}{Video}.}
    \label{fig:inversive}
\end{figure}

\subsection*{Summary of the Article:}

\begin{itemize}
    \item \cref{sec:prelims}: We review basic relations for the elliptic billiard.
    \item \cref{sec:n3-inv-billiard}: The focus-inversive family derived from billiard 3-periodics is non-Ponceletian, and its caustic can be compact and regular but also contain cusps, and/or be non-compact. A special triangle center \cite[Gergonne Point]{mw} is stationary. Both perimeter and sum of cosines are invariant. Indeed, focus-inversive triangles are also billiard 3-periodics, but of a rigidly-moving billiard table, whose center moves along a circle. The sum of distances from the inversion center (a focus) to each vertex is invariant as is the product of areas of left- and right-focus-inversive triangles.
    \item \cref{sec:n3-loci}: the loci of 29 of the first 100 centers listed in \cite{etc} are {\em circles}, and these include the vertex, area, and perimeter centroids. ``Why so many circles?'' is an unanswered question. We derive expressions for some centers and radii of said loci. Indeed, based on experiments with 1000 triangle centers, we conjecture that if the locus of a triangle center in the billiard family is a conic, then it is a circle in the inversive family. Proofs are welcome.
    \item \cref{sec:n3-ctr-inv}: we study a closely-related, ``center-inversive'' family, whose vertices are inversions of billiard 3-periodics with respect to a circle concentric with the billiard. The family is also non-Ponceletian (inscribed in Booth's curve and circumscribing a non-conic caustic). We prove that the locus of the circumcenter is an ellipse; based on experiments, we conjecture said center to be the only one capable of producing a conic locus.
\end{itemize}

In \cref{app:invariants} we review some classical quantities of the elliptic billiard.
%In Appendix~\ref{app:vertices-caustics} we provide explicit expressions for vertices and caustics of 3-periodics.
A reference table with all symbols used appears in \cref{app:symbols}. A link to a YouTube animation is provided in the caption of most figures. \cref{tab:playlist} in \cref{sec:videos} compiles all videos mentioned. 

\subsection*{Related Work}

Experimentation with the dynamic geometry of 3-periodics in the elliptic billiard evinced that the loci of the incenter, barycenter, and circumcenter are ellipses. These observations were soon proved \cite{garcia2019-incenter,olga14,schwartz2016-com}. Indeed, using a combination of numerical and computer-algebra methods, we showed that the loci of a full 29 out of the first 100 centers listed in Kimberling's Encyclopedia \cite{etc} are ellipses \cite{garcia2020-ellipses}. Odehnal's studied the loci of triangle centers in the poristic family \cite{odehnal2011-poristic}, detecting many as stationary (points), ellipses, and circles.

\section{Preliminaries}
\label{sec:prelims}
Throughout this article we assume the elliptic billiard is the ellipse:

\begin{equation*}
f(x,y)=\left(\frac{x}{a}\right)^2+\left(\frac{y}{b}\right)^2=1,\;\;\;a>b>0.
%\label{eqn:billiard-f}
\end{equation*}

Let $\delta=\sqrt{a^4 - a^2 b^2 - b^4}$ and $\rho$ denote the radius of the inversion circle. The eccentricity $\varepsilon=\sqrt{1-(b/a)^2}$, and $c^2=a^2-b^2$.

Recall the definition of a {\em triangle center}: these are points on the plane of a triangle defined by functions of their sidelenghts and/or angles \cite{kimberling1993_rocky}. Centers will be referred to using the $X_k$ notation, $X_1$ is the incenter, $X_3$ the circumcenter, in the manner of in Kimberling's Encyclopedia \cite{etc}. Any symbols suffixed by $\dagger$ (resp. $\odot$) refer to the focus-inversive (resp. center-inversive) triangle, e.g., $X_9^\dagger$ is the mittenpunkt of the focus-inversive triangle, $L^\dagger$ its perimeter, etc. 

\subsection*{A word about our proof method}

We omit most proofs as they have been produced by a consistent process, namely: (i) using an explicit expressions for billiard 3-periodic vertices (see \cite{garcia2019-incenter,reznik2020-ballet}), normally setting a first vertex $P_1=(a,0)$ so as to obtain an isosceles configuration; (ii) obtain a symbolic expression for the invariant of interest; (iii) simplify it (both human intervention and CAS); finally (iv) verify its validity symbolically and/or numerically over several N-periodic configurations and elliptic billiard aspect ratios $a/b$.

\section{Focus-Inversives and the Rotating Billiard Table}
\label{sec:n3-inv-billiard}
A generic $N>3$ focus-inversive polygon is illustrated in \cref{fig:inversive}. This is a polygon whose vertices are inversions of the N-periodic vertices wrt a circle of radius $\rho$ centered on focus $f_1$.

It is well known the inversion of an ellipse with respect to a focus-centered unit circle is a loopless Limaçon of Pascal $\mathcal{L}(t)$ \cite{ferreol2020-limacon}, given by:

\[ \mathcal{L}(t)=\left[-a \cos{t} (\varepsilon  \cos{t}+1)-c,a \sin{t} (\varepsilon \cos{t}+1)\right],\;\;0{\leq}t{\leq}2\pi
\]
\noindent Therefore:

\begin{remark}
For any $N$, the focus-inversive family is inscribed in $\mathcal{L}(t)$.
\end{remark}

\cref{fig:n3-caustics} shows the non-conic ``caustic'' (envelope of sidelines) to the $N=3$ focus-inversive family for various values of $a/b$. As $a/b$ is increased, the caustic transitions from (i) a regular curve, to (ii) one with a self-intersection and two cusps, to (iii) a non-compact curve with two infinite branches.

\cref{fig:n3-inv} illustrates the $N=3$ focus-inversive triangle. Recall two well-known triangle centers: (i) the Gergonne point $X_7$: this is where a triangle and its intouch (contact) triangle are perspective, see \cite[Gergonne Point]{mw}; (ii) the Mittenpunkt $X_9$: where lines drawn from the excenters through the side midpoints concur \cite[Mittenpunkt]{mw}. 

\begin{figure}
    \centering
    \includegraphics[width=\textwidth]{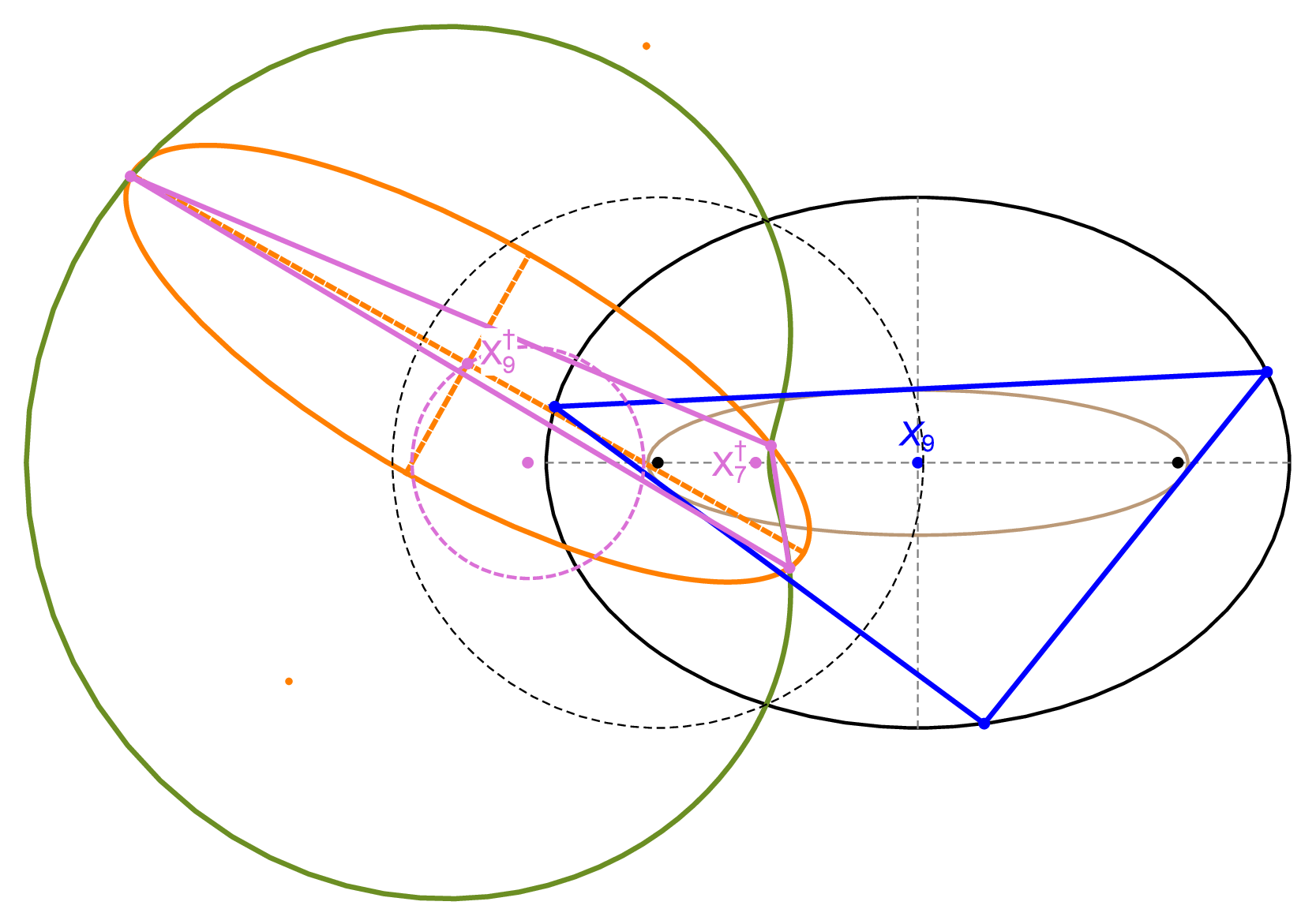}
    \caption{The vertices of 3-periodics (blue), when inverted with respect to a unit circle (dashed black) centered on a focus, produce a family of constant perimeter triangles (pink). Their $X_9$-centered circumellipses (orange) rigidly translate and rotate (invariant semi-axes) with their center $X_9^\dagger$ moving along a circle. The Gergonne point $X_7^\dagger$ of the inversive family is stationary. \href{https://youtu.be/LOJK5izTctI}{Video 1}, \href{https://youtu.be/Y-j5eXqKGQE}{Video 2}}
    \label{fig:n3-inv}
\end{figure}

\begin{figure}
    \centering
    \includegraphics[width=\textwidth]{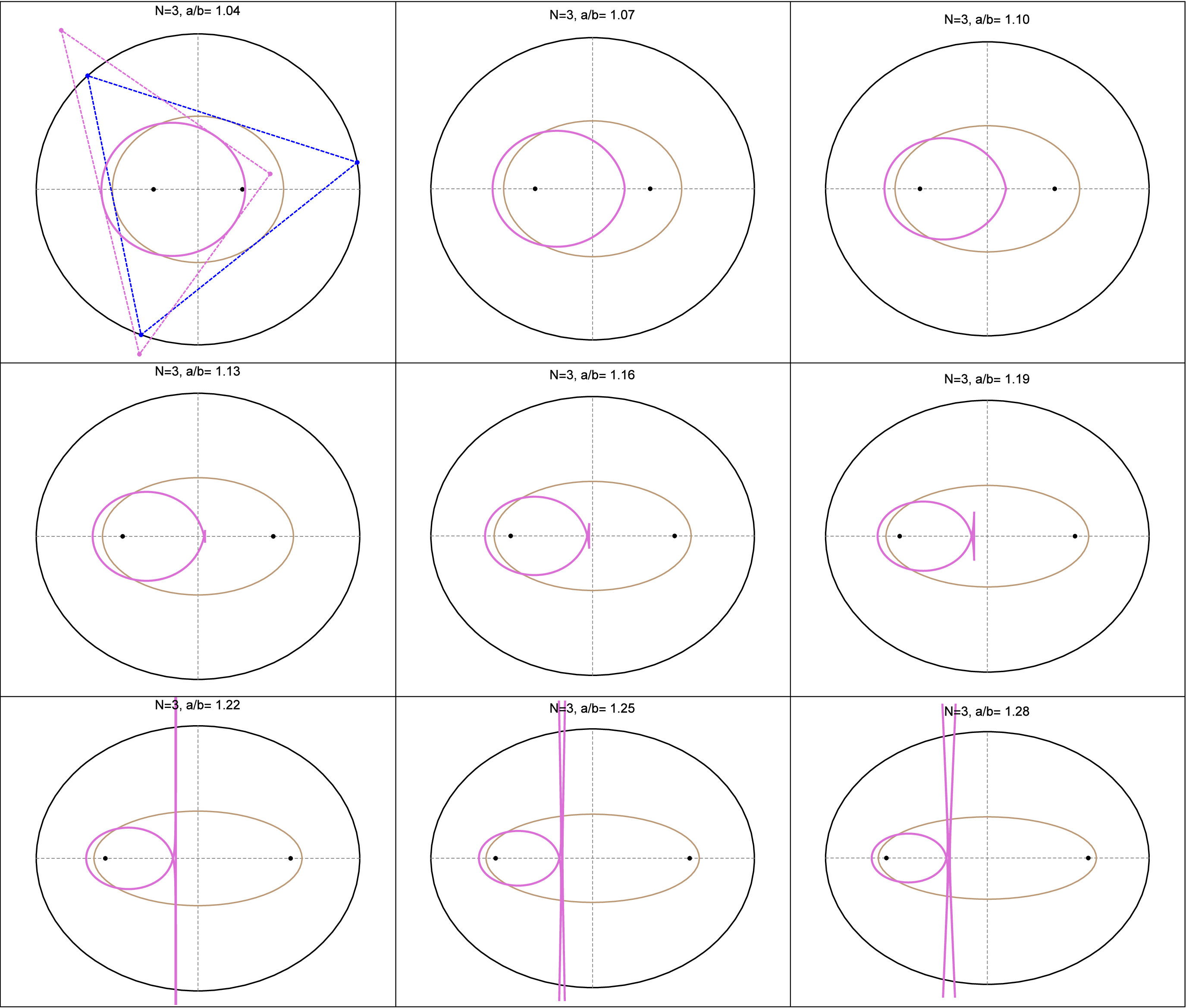}
    \caption{Non-conic caustic (pink) to the focus-inversive family (pink). As $a/b$ increases, said caustic transitions from (i) regular, to (ii) a curve with one self-intersection and two cusps, to (iii) a non-compact curve.  A sample 3-periodic (dashed blue) and the corresponding focus-inversive triangle (dashed pink) is shown on the first frame only. The billiard confocal caustic (brown) is shown on every frame. \href{https://bit.ly/374jbBl}{App}}
    \label{fig:n3-caustics}
\end{figure}

\begin{proposition}
Over the 3-periodic family, the Gergonne point $X_7^\dagger$ of focus-inversive triangles is stationary on the billiard's major axis at location:
\[ X_7^\dagger=\left[c\left(1-\frac{\rho^2}{\delta+c^2}\right),0\right]\]
\end{proposition}

Note: let $X_7^\ddagger$ denote the inversion of $X_7^\dagger$ with respect to the same inversion circle used above. Then:

\[ X_7^\ddagger=\left[\frac{\delta}{c},0\right]\]

\begin{proposition}
Over the 3-periodic family, the Mittenpunkt $X_9^\dagger$ of focus-inversive triangles moves along a circle with center and radius given by:
\begin{align*}
    C_9^\dagger=&\left[-c\left(1+\rho^2 \frac{1}{2b^2}\right), 0\right] \\
    R_9^\dagger=&\rho^2\, \frac{2a^2-b^2-\delta}{2 a b^2}
\end{align*}
\end{proposition}

\begin{remark}
Let $C_9^\ddagger$ denote the inversion of $C_9^\dagger$ with respect to the same inversion circle used above. Then:

\[ C_9^\ddagger=\left[-\frac{a^2+b^2}{c}, 0\right]\]
\end{remark}

\begin{proposition}
The perimeter $L^\dagger$ of the $N=3$ focus-inversive family is invariant and given by:

\[L^\dagger=\rho^2 \frac {\sqrt { \left( 8\,{a}^{4}+4\,{a}^{2}{b}^{2}+2\,{b}^{4}
 \right) \delta+8\,{a}^{6}+3\,{a}^{2}{b}^{4}+2\,{b}^{6}}}{{a}^{2}{b}^{
2}}\]
%\label{prop:n3-per}
\end{proposition}

\noindent Surprisingly, perimeter invariance of the focus-inversive polygon has been generalized to all $N$ \cite{roitman2021-bicentric}.

\begin{proposition}
For $N=3$, the sum of cosines of focus inversive triangles is given by:
\[\sum\cos{\theta_{1,i}^\dagger}=\frac{\delta (a^2+c^2-\delta)}{a^2c^2}\]
\end{proposition}

As seen above, the focus-inversive family simultaneously conserves perimeter and sum of cosines. Still referring to \cref{fig:n3-inv}:

\begin{corollary}
With respect to a reference system centered on $X_9^\dagger$ and oriented along the semi-axes of $\mathcal{C}^\dagger$, the inversive $N=3$ family is a 3-periodic billiard family of $\mathcal{C}^\dagger$.
\end{corollary}

Indeed, $\mathcal{C}^\dagger$ can be regarded as the inversive triangle's rigidly-moving mittenpunkt-centered circumconic, called the ``circumbilliard'' in \cite{reznik2021-circum}:

\begin{proposition}
The $X_9^\dagger$-centered circumellipse $\mathcal{C}^\dagger$ of the $N=3$ inversive family has invariant semi-axes $a^\dagger$ and $b^\dagger$ given by:

\begin{align*}
    a^\dagger&= \rho\,k_1
\sqrt {k_2\left(\delta+a\,c\right)}\\
    b^\dagger&= \rho\,k_1
\sqrt {k_2\left(\delta-a\,c\right) }\\
 \mbox{where:}&\\
k_1&=\frac{c \sqrt{2}}{k_3}\sqrt { \left( 8\,{a}^{4}+4\,{a}^{2}{b}^{2}+2\,{b}^{4}
 \right) \delta+8\,{a}^{6}+3\,{a}^{2}{b}^{4}+2\,{b}^{6}}\\
 k_2&=2 a^2-b^2-\delta\\
 %\sqrt{ \left( -4\,{a}^{2}+2\,{b}^{2} \right) \delta+5\,{a}^{4}-5\,{a}^{2}{b}^{2}+2\,{b}^{4}}\\
k_3&=2a  {b}^{2} \left(  \left( 2\,{a}^{2}-{b}^{2} \right) \delta+2
\,{a}^{4}-2\,{a}^{2}{b}^{2}-{b}^{4}\right)
\end{align*}

\end{proposition}

\subsection*{More Focus-Inversive Invariants}

As noted in \cref{tab:symbols}, let $d_{1,i}=|P_i-f_1|$.

\begin{proposition}
For $N=3$, the sum of focus-inverse spoke lengths is invariant and given by:

\[  \sum\frac{1}{d_{1,i}}=\frac {{a}^{2}+{b}^{2}+\delta}{a{b}^{
2}}
\]
%= {\frac {J\sqrt{2}\sqrt{ J\,L+\sqrt {9-2\,JL}-3 }}{ J L-4 }}
% =\left(\frac{2J}{JL-4}\right)^2

\end{proposition}

\noindent Referring to \cref{fig:inv-pedals}:

\begin{proposition}
For $N=3$, the area product $A_1^\dagger A_2^\dagger$ of the two focus-inversive triangles is given by:

\[ A_1^\dagger A_2^\dagger= \frac{\rho^8}{8 a^8 b^2}  \left[\left( {a}^{4}+2\,{a}^{2}{b}^{2}+4\,{b}^{4} \right)\delta +  \frac{3 a^{4} b^{2}}{2}+a^6+4\,{b}^{6} \right]
\]
\end{proposition}

Experiments suggest:

\begin{conjecture}
Over billiard $N$-periodics, $N{\geq}3$ is odd, the area product $A_1^\dagger A_2^\dagger$ is invariant.
\end{conjecture}

\begin{observation}
The areas of both pedal triangles of the focus-inversives with respect to the respective foci are dynamically equal.
\end{observation}

\begin{figure}
    \centering
    \includegraphics[width=\textwidth]{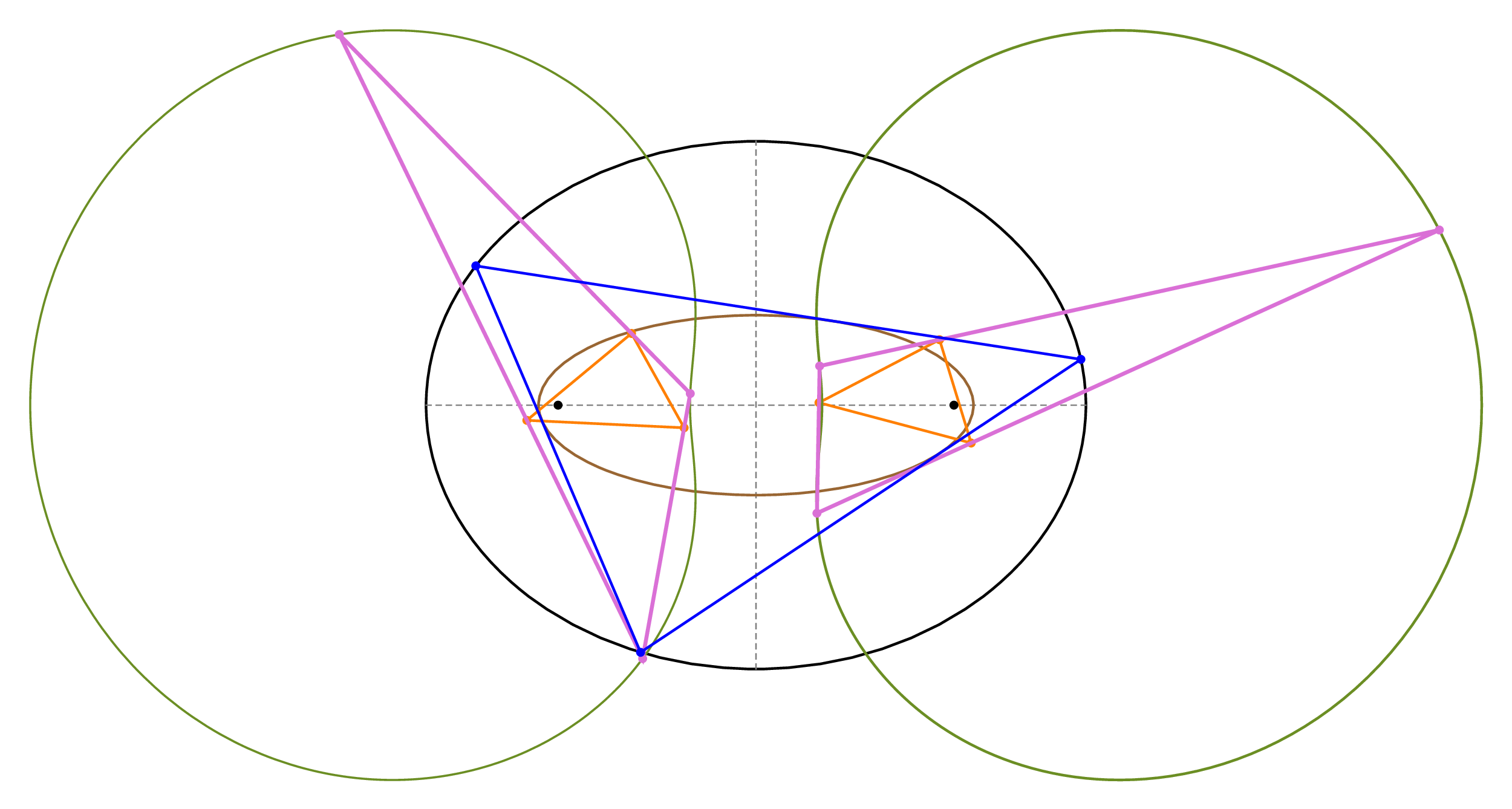}
    \caption{$f_1$- and $f_2$-inversive triangles are shown (pink). The product of their areas is invariant. Also shown are the $f_1$- and $f_2$-pedal polygons (orange) to said triangles. Their areas are variable but dynamically equal. \href{https://youtu.be/0L2uMk2xyKk}{Video}}
    \label{fig:inv-pedals}
\end{figure}

\section{Circular Loci of Triangle Centers}
\label{sec:n3-loci}
In \cite{garcia2020-ellipses} we showed that amongst the first 100 triangle centers listed on \cite{etc}, the following 29 sweep an elliptic locus over 3-periodics in the elliptic billiard: $X_k$, $k$=1, 2, 3, 4, 5, 7, 8, 10, 11, 12, 20, 21, 35, 36, 40, 46, 55, 56, 57, 63, 65, 73, 78, 79, 80, 84, 88, 90, 100. Notably, the Mittenpunkt $X_9$ is stationary at the billiard center \cite{reznik2020-intelligencer}.

Referring to \cref{fig:n3-loci-12345}:

\begin{theorem}
Amongst the first 100 triangle centers listed on \cite{etc}, the following 28 sweep a circular locus over focus-inversive 3-periodics in the elliptic billiard: $X_k^\dagger$, $k$=1, 2, 3, 4, 5, 8, 9, 10, 11, 12, 20, 21, 35, 36, 40, 46, 55, 56, 57, 63, 65, 73, 78, 79, 80, 84, 90, 100. All centers lie on the billiard major axis. 
\end{theorem}

\begin{proof}
Symbolic manipulation and numeric verification.
\end{proof}

Through painstaking CAS-assisted simplification, we were able to obtain compact expressions for only a few of the above circular loci, namely: $X_k,k=$1, 2, 3, 4, 5, 9, 11, 100. 

\begin{proposition}
The locus of $X_1^\dagger$ is the circle given by:
\begin{align*}
C_1^\dagger=&\left[c\left(-1+\rho^2\frac{-2a^2+b^2+2\delta}{2b^4}\right), 0\right]\\
R_1^\dagger=&\rho^2\frac{-2\delta^2+b^4+(2a^2-b^2)\delta}{2ab^4}
\end{align*}
\end{proposition}

\begin{proposition}
The locus of $X_2^\dagger$ is the circle given by:

\begin{align*}
C_2^\dagger&=\left[-c\left(1+\rho^2\frac{2a^2-b^2-\delta}{3 a^2 b^2}\right),0\right]\\
R_2^\dagger&=\rho^2\frac{2 a^2-b^2-\delta}{3 a b^2}
\end{align*}
\end{proposition}

\begin{proposition}
The locus of $X_3^\dagger$ is the circle given by:
\begin{align*}
C_3^\dagger=&\left[-c\left(1+\rho^2\frac{a^2+b^2}{2b^4}\right), 0\right]\\
R_3^\dagger=&\rho^2\frac{a(-b^2+\delta)}{2b^4}
\end{align*}
\end{proposition}

\begin{proposition}
The locus of $X_4^\dagger$ is the circle given by:
\begin{align*}
C_4^\dagger=&\left[c\left(-1+\rho^2\frac{(b^2+\delta) \delta}{a^2 b^4}\right), 0\right]\\
R_4^\dagger=&\rho^2\frac{c^2(b^2+\delta)}{a b^4}
\end{align*}
\end{proposition}

\begin{proposition}
The locus of $X_5^\dagger$ is the circle given by:
\begin{align*}
C_5^\dagger=&\left[c\left(-1+ \rho^2\frac{a^4 - 3 a^2 b^2+2 b^4 + 2 b^2 \delta}{4 a^2 b^4}\right), 0\right]\\
R_5^\dagger=&\rho^2\frac{(3a^2-2b^2)b^2+(a^2-2b^2)\delta}{4a b^2}
\end{align*}
\end{proposition}

\begin{proposition}
The locus of $X_{11}^\dagger$ is the circle given by:

\begin{align*}
C_{11}^\dagger&=\left[c\left(-1+\rho^2\frac{-a^2+b^2+\delta }{2 a^2 b^2}\right),0\right]\\
R_{11}^\dagger&=\rho^2\frac{-a^2+b^2+\delta}{2 a b^2}
\end{align*}
\end{proposition}

\begin{proposition}
The locus of $X_{100}^\dagger$ is the circle given by:

\begin{align*}
C_{100}^\dagger&=\left[-c\left(1+\rho^2\frac{1}{b^2}\right),0\right]\\
R_{100}^\dagger&=\rho^2\frac{a}{b^2}
\end{align*}
\end{proposition}

As noted above, the focus-inversive Gergonne point $X_7^\dagger$ is stationary on the major axis.

\begin{figure}
    \centering
    \includegraphics[width=\textwidth]{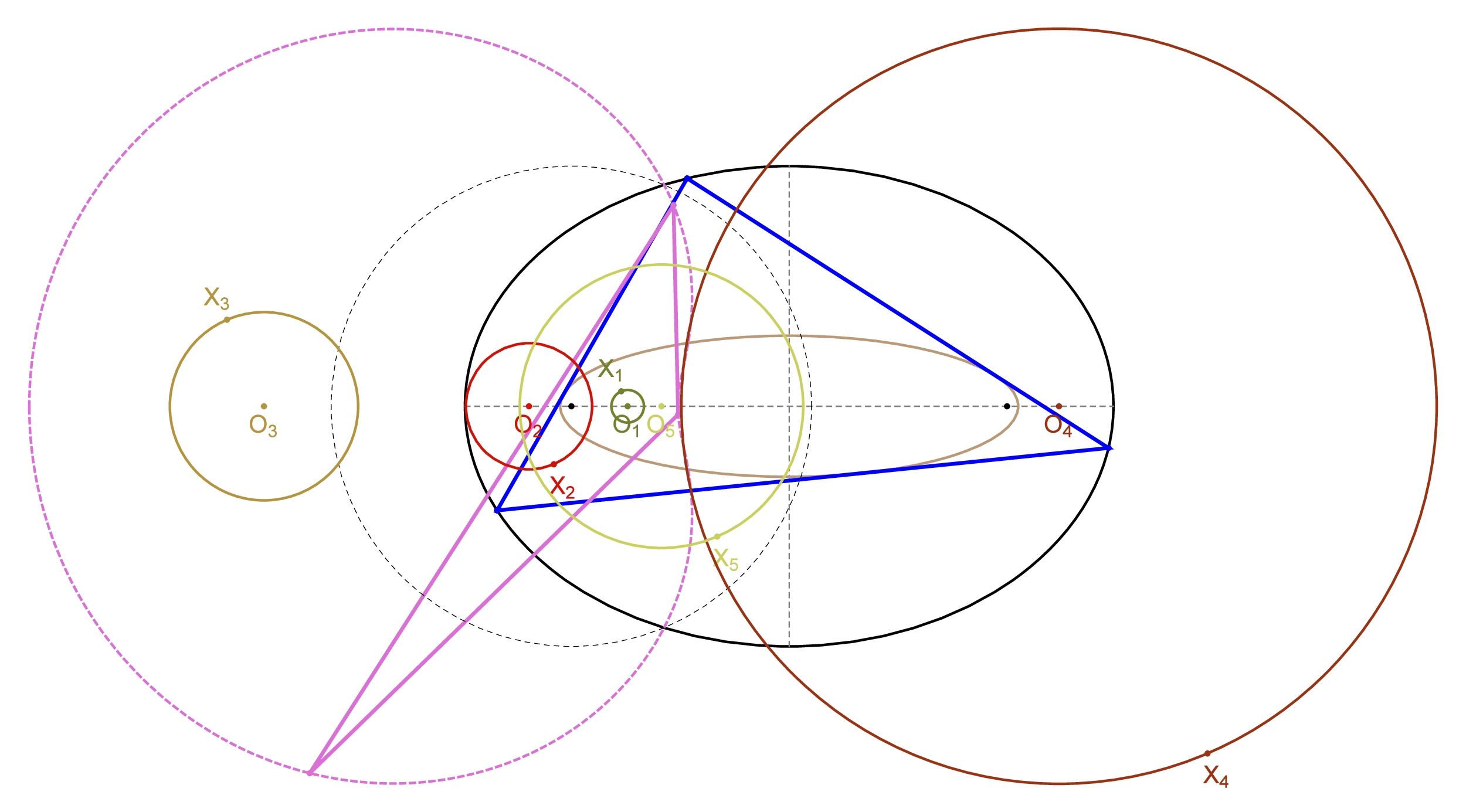}
    \caption{The focus-inversive 3-periodic (pink) is shown inscribed in Pascal's Limaçon (dashed pink). Also shown are the circular loci of $X_k^\dagger$, $k=1,2,3,4,5$. Notice all its centers $O_i$ lie on the billiard's major axis. \href{https://youtu.be/OAD2hpCRgCI}{Video}}
    \label{fig:n3-loci-12345}
\end{figure}

In \cite{reznik2020-ballet} triangle centers whose locus over 3-periodics was the billiard boundary were called ``swans''. As it was shown, both $X_{88}$ and $X_{100}$ are swans \cite{reznik2020-ballet}. Curiously, the former's (resp. latter's) motion is non-monotonic (monotonic) with respect to a continuous clockwise parametrization of 3-periodics. 

Referring to \cref{fig:x159-x934}:

\begin{observation}
Amongst all 29 triangle centers whose loci are ellipses over billiard 3-periodics, only $X_{88}^\dagger$ does not sweep a circular locus over the focus-inversive family. Nevertheless, this locus is algebraic, regular, can be both convex and concave, and is symmetric wrt the major axis of the ellipse.
\end{observation}

Referring to \cref{fig:non-elliptic-swans}, $X_{162}$ is yet another example of a billiard ``swan'' whose focus-inversive locus is non-elliptic.

\begin{figure}
    \centering
    \includegraphics[width=.8\textwidth]{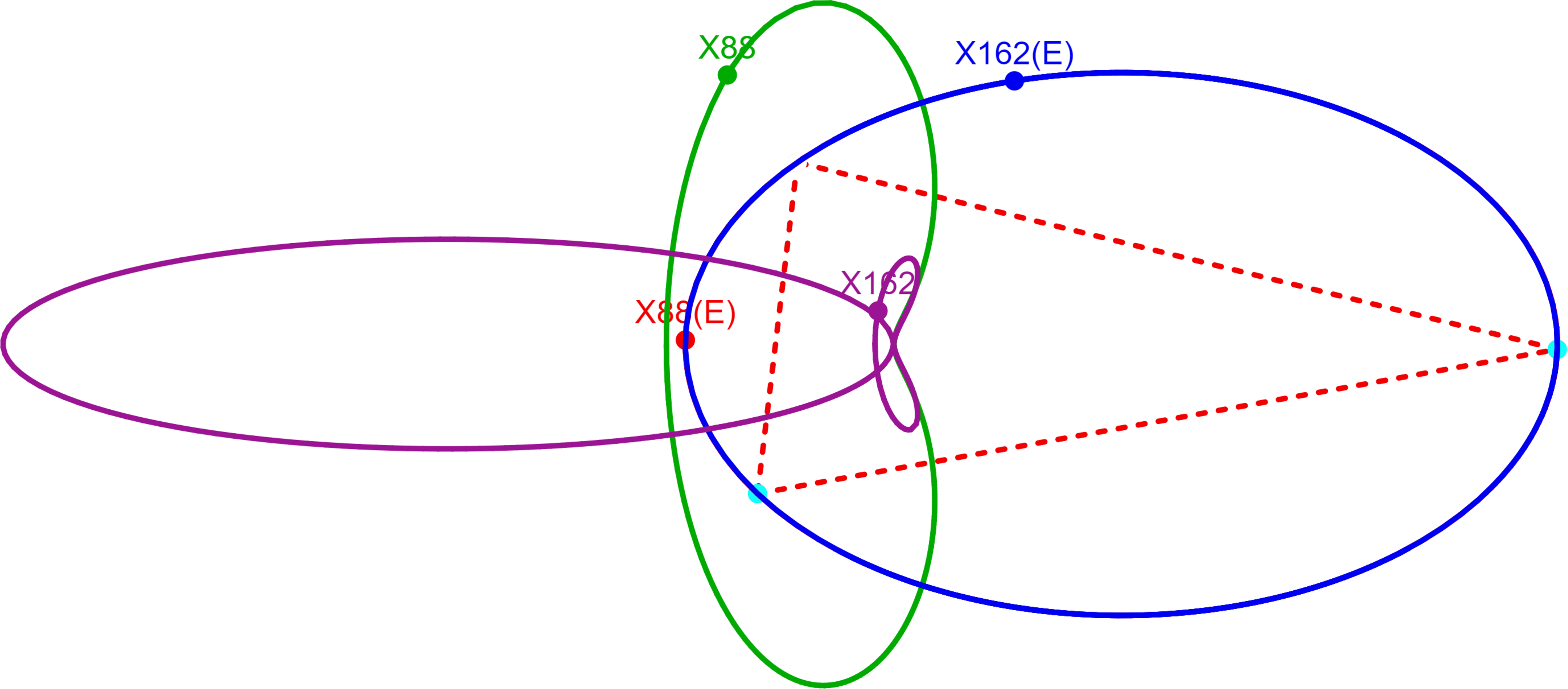}
    \caption{Over billiard 3-periodics (dashed red) the loci of both $X_{88}$ and $X_{162}$  coincide with the billiard (blue). However, when taken as centers of the the focus-inversive triangles (not shown), their loci are clearly non-elliptic (green and purple). Live app: \href{https://bit.ly/32XfPyZ}{X88+X162}, \href{https://bit.ly/3lF3TsO}{X88+X100}}
    \label{fig:non-elliptic-swans}
\end{figure}

\begin{observation}
For $101{\leq}k{\leq}1000$, the following 68 focus-inversive triangle centers also sweep circular loci: $X_k^\dagger$, $k=$104, 119, 140, 142, 144, 145, 149, 150, 153, 165, 191, 200, 210, 214, 224, 226, 329, 354, 355, 376, 377, 381, 382, 388, 390, 392, 404, 405, 411, 442, 443, 452, 474, 480, 484, 495, 496, 497, 498, 499, 546, 547, 548, 549, 550, 551, 553, 631, 632, 908, 920, 934, 936, 938, 942, 943, 944, 946, 950, 954, 956, 958, 960, 962, 993, 997, 999, 1000.
\end{observation}

Referring to \cref{fig:x159-x934}:

\begin{observation}
For $k{\leq}1000$, the only triangle centers whose billiard locus is not an ellipse and whose focus-inversive locus is a circle are $X_{150}$ and $X_{934}$. In fact, the latter's locus is identical to $X_{100}^\dagger$.
\end{observation}

\begin{figure}
    \centering
    \includegraphics[width=\textwidth]{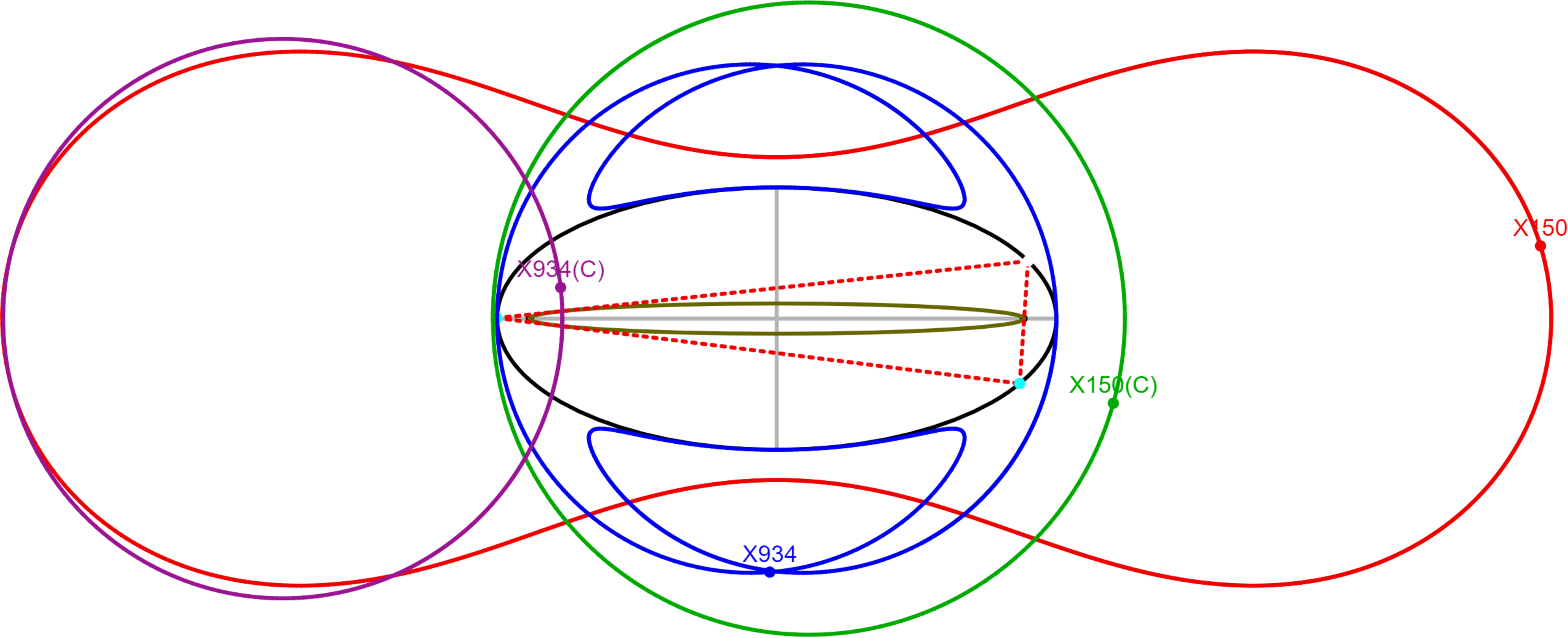}
    \caption{The locus of $X_{150}$ is a dumbbell curve (red). Curiously, the locus of $X_{150}$ over the focus inversives is a circle (green). Similarly, the billiard locus of $X_{934}$ (blue) is a curve with two self-intersections; its locus over the focus-inversives is a circle. \href{https://bit.ly/2IQY65u}{App}}
    \label{fig:x159-x934}
\end{figure}

Experimentally, in the range $k{\leq}1000$, if the locus of $X_k$ is an ellipse over billiard 3-periodics (excluding the cases where the locus is the billiard itself), then the locus of $X_k^\dagger$ over the focus-inversive family is a circle. Therefore:

\begin{conjecture}
If the the locus of some triangle center $X$ is an ellipse over billiard 3-periodics, then the locus of $X^\dagger$ over the inversive family is a circle.
\end{conjecture}

\subsection{Locus of Perimeter Centroid}

Let $C_0$, $C_1$, $C_2$ denote the vertex, perimeter, and area centroids of polygon, respectively. In \cite[Thm 1]{schwartz2016-com} it was shown that the loci of $C_0,C_2$ over Poncelet families are ellipses, though this not hold in general for $C_1$.

For triangles, $C_0=C_2=X_2$ and $C_1=X_{10}$ \cite[Spieker Center]{mw}. Per above we already know that the loci of $X_{2}$ and $X_{10}$ over the focus-inversive family are circles. Therefore, and referring to \cref{fig:n3-centroids}:

\begin{corollary}
The loci of the $C_i^\dagger$, $i=1,2,3$ of the focus-inversive family are circles.
\end{corollary}

\begin{figure}
    \centering
    \includegraphics[width=.7\textwidth]{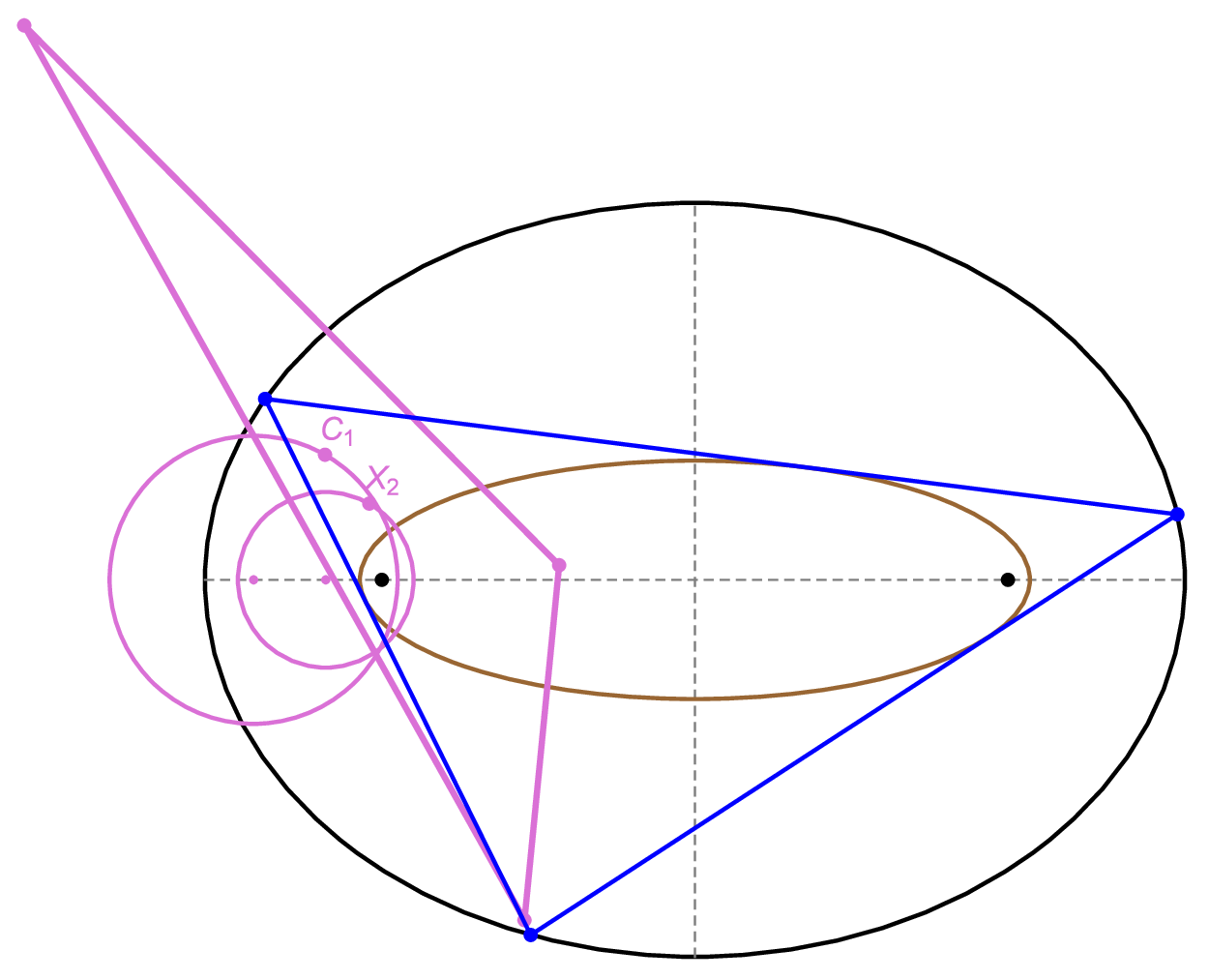}
    \caption{Circular locus of the focus-inversive $X_2^\dagger$ and the perimeter centroid $C_1^\dagger=X_{10}^\dagger$. Note that for triangles, the former coincides with both the vertex and area centroids. \href{https://bit.ly/3kYnGlP}{App}}
    \label{fig:n3-centroids}
\end{figure}

Surprisingly, experiments suggest:

\begin{conjecture}
Over the focus-inversive family derived from billiard $N$-periodics, $N{\geq}3$, the locus of vertex, perimeter, and area centroids are all circles.
\end{conjecture}

\section{The Center-Inversive Family}
\label{sec:n3-ctr-inv}
Let the {\em center-inversive triangle} have vertices 
%$P_i^\odot$
at the inversions of the $P_i$ w.r.t. unit-radius circle concentric with the elliptic billiard; see \cref{fig:center-inversive}. Therefore this family is inscribed in Booth's curve (the inverse of an ellipse with respect to a concentric circle) \cite{ferreol2020-booth}. As shown in \cref{fig:n3-ctr-caustics}, as $a/b$ increases, the caustic transitions from a convex-regular curve, to one with two self-intersections and cusps, to a non-compact one. 

\begin{proposition}
The locus of the circumcenter $X_3^\odot$ of the center-inversive triangle is an ellipse concentric, axis-aligned, and homothetic to the elliptic locus of $X_3$ of the billiard 3-periodics, at ratio $1/\delta$. Moreover, the aspect ratio of the locus of $X_3^\odot$ is the reciprocal of that of the caustic.
\end{proposition}

\begin{proof}
The circumcircle of the center-inversive triangle is the inversion of the circumcircle of the billiard's 3-periodic. Denote by $R$ and $X_3$ the circumradius and circumcenter of the 3-periodic, and by $X_3^\odot$ the circumcenter of the center-inversive triangle. By properties of circle inversion, we have that the center of the billiard $O$, $X_3$, and $X_3^\odot$ are collinear, with $O X_3^\odot=\frac{O X_3}{{O X_3}^2-R^2}$ (recall that the inversion of the center of a circle does not coincide with the center of the inverted circle). The quantity ${O X_3}^2-R^2$ represents the power of point $O$ with respect to the circumcircle of the 3-periodic. In \cite[Thm 3]{garcia2020-new-properties}, it was shown that the power of the billiard center with respect to this circumcircle is constant over the family of 3-periodics and equal to $-\delta$. Moreover, in \cite{garcia2019-incenter}  it was proved that the locus of $X_3$ is an ellipse axis-aligned and concentric with the billiard and similar to the caustic (rotated by $\pi/2$). Thus, the locus of $X_3^\odot$ must be an ellipse concentric, axis-aligned, and homothetic to that one at a ratio of 1/$\delta$.
\end{proof}

Experimentally, out of $X_k$, $k=1{\ldots}1000$, only $X_3$ sweeps a conic locus over the center-inversive family. Therefore:

\begin{conjecture}
The circumcenter $X_3$ is the only triangle center whose locus over the center-inversive family is a conic.
\end{conjecture}

\begin{figure}
    \centering
    \includegraphics[width=.7\textwidth]{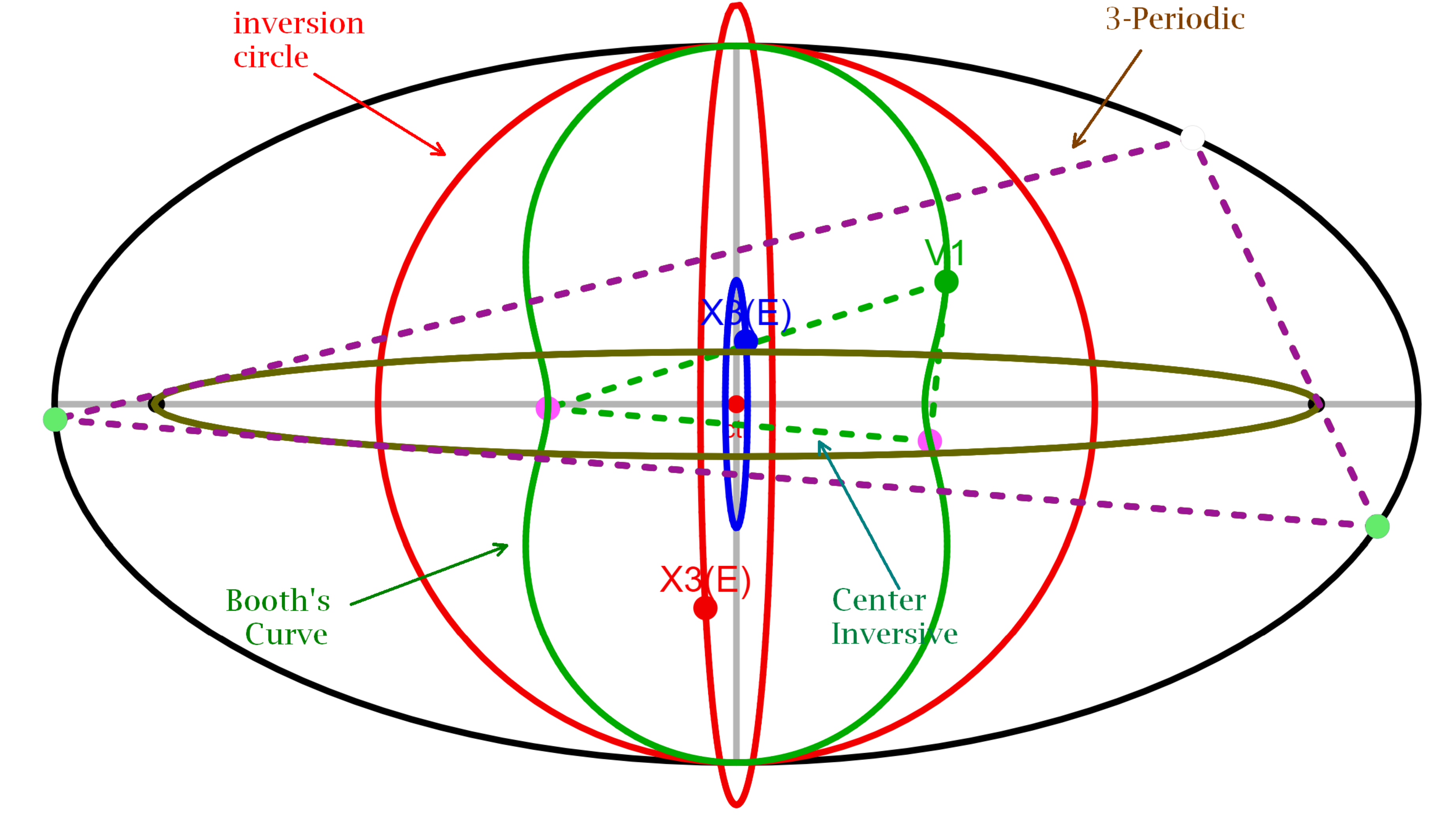}
    \caption{The center-inversive triangle (green) has vertices at the inversions of the 3-periodic vertices wrt unit-circle (red) concentric with the ellipse. It is inscribed in Booth's curve (green) \cite{ferreol2020-booth}. Over the 3-periodics, the locus of $X_3$ is an ellipse (red) \cite{garcia2019-incenter}. The locus of $X_3$ of the center-inversives is also an ellipse (blue), concentric, axis-aligned, and homothetic to the former at ratio $1/\delta$. \href{https://bit.ly/3lZbzX1}{App}}
    \label{fig:center-inversive}
\end{figure}

\begin{figure}
    \centering
    \includegraphics[width=\textwidth]{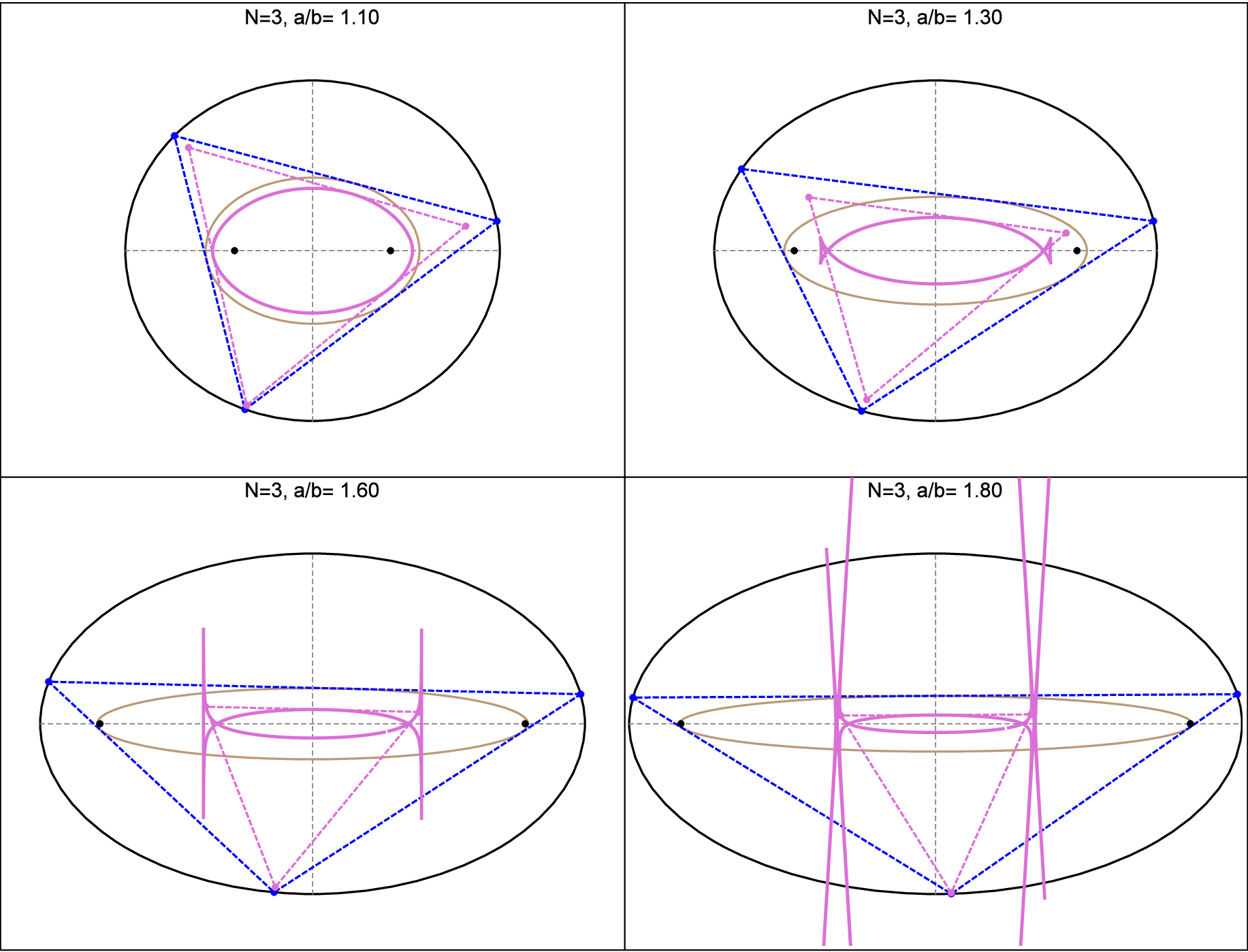}
    \caption{Non-conic caustic (pink) to the center-inversive family (dashed pink). As $a/b$ increases it transitions from (i) a regular-convex curve, to (ii) one with two self-intersections and four cusps, to (iii) a non-compact, self-intersected curve. \href{https://bit.ly/2HAk2kC}{App1}, \href{https://bit.ly/39aT3Y5}{App2}}
    \label{fig:n3-ctr-caustics}
\end{figure}

%\clearpage
\section{Videos}
\label{sec:videos}
Animations illustrating some focus-inversive phenomena are listed on \cref{tab:playlist}.

\begin{table}[H]
\small
\begin{tabular}{|c|l|l|}
\hline
id & Title & \textbf{youtu.be/<...>}\\
\hline
01 & {N=5 invariants} &
\href{https://youtu.be/wkstGKq5jOo}{\texttt{wkstGKq5jOo}}\\
02a & {N=3 circumbilliard I} & 
\href{https://youtu.be/LOJK5izTctI}{\texttt{LOJK5izTctI}}\\
02b & {N=3 circumbilliard II} & 
\href{https://youtu.be/Y-j5eXqKGQE}{\texttt{Y-j5eXqKGQE}}\\
03 & {N=3 invariant area product} & 
\href{https://youtu.be/0L2uMk2xyKk}{\texttt{0L2uMk2xyKk}}\\
04 & {N=5 invariant area product} &
\href{https://youtu.be/bTkbdEPNUOY}{\texttt{bTkbdEPNUOY}}\\
05 & {N=3 equal-area pedals} &
\href{https://youtu.be/0L2uMk2xyKk}{\texttt{0L2uMk2xyKk}}\\
\hline
06a & {N=3 circular loci I} &
\href{https://youtu.be/OAD2hpCRgCI}{\texttt{OAD2hpCRgCI}} \\
06b & {N=3 circular loci II} &
\href{https://youtu.be/tKB-50zW8F4}{\texttt{tKB-50zW8F4}} \\
06c & {N=3 circular loci III} &
\href{https://youtu.be/srjm23nQbMc}{\texttt{srjm23nQbMc}}\\
\hline
\end{tabular}
\caption{Videos of some focus-inversive phenomena. The last column is clickable and provides the YouTube code.}
\label{tab:playlist}
\end{table}

\section*{Acknowledgments}
We would like to thank Arseniy Akopyan, Peter Moses, and Pedro Roitman for useful insights.

The second author is fellow of CNPq and coordinator of Project PRONEX/ CNPq/ FAPEG 2017 10 26 7000 508.

\appendix

\section{Review: Elliptic Billiard}
\label{app:invariants}
Joachimsthal's Integral expresses that every trajectory segment is tangent to a confocal caustic \cite{sergei91}. Equivalently, a positive quantity $J$ remains invariant at every bounce point $P_i=(x_i,y_i)$:

\begin{equation*}
 J=\frac{1}{2}\nabla{f_i}.\hat{v}=\frac{1}{2}|\nabla{f_i}|\cos\alpha
 %\label{eqn:joachim}
\end{equation*}

\noindent where $\hat{v}$ is the unit incoming (or outgoing) velocity vector, and:

\begin{equation*}
\nabla{f_i}=2\left(\frac{x_i}{a^2}\,,\frac{y_i}{b^2}\right).
%\label{eqn:fnable}
\end{equation*}

Hellmuth Stachel contributed \cite{stachel2020-private} an elegant expression for Joachmisthal's constant $J$ in terms of EB semiaxes $a,b$ and the major semiaxes $a''$ of the caustic:

\begin{equation*}
    J = \frac{\sqrt{a^2 - a''^2}}{{a}{b}}
\end{equation*}

\noindent The signed area of a polygon is given by the following sum of cross-products \cite{preparata1988}:

\[  A=\frac{1}{2}\sum_{i=1}^N
%\det\begin{bmatrix} 
%x_{i+1} & x_{i} & x_{i-1} \\
%y_{i+1} & y_{i} & y_{i-1} \\
%1 & 1 & 1 \\
%\end{bmatrix}
%\quad
{(P_{i+1}-P_{i})\times(P_i-P_{i+1})} \]

Let $d_{j,i}$ be the distance $|P_i-f_j|$. The inversion $P_{j,i}^\dagger$ of vertex $P_i$ with respect to a circle of radius $\rho$ centered on $f_j$ is given by:

\[ P_{j,i}^\dagger=f_j+ \left(\frac{\rho}{d_{j,i}}\right)^2 (P_i-f_j)\]

\subsection*{N=3 case}

For $N=3$ the following explicit expressions for $J$ and $L$ have been derived \cite{garcia2020-new-properties}:

\begin{align*}
J=&\frac{\sqrt{2\delta-a^2-b^2}}{c^2} \\ %\nonumber\\
L=&2(\delta+a^2+b^2)J %\label{eqn:lj}
\end{align*}

\noindent When $a=b$, $J=\sqrt{3}/2$ and when $a/b{\to}\infty$, $J{\to}0$. 

%\section{Vertices \& Caustics N=3}
%\label{app:vertices-caustics}
%\input{220_app_n3}

%\section{Focus-Inversive Gergonne Point}
%\label{app:gergonne}
%\input{225_app_x7}

\section{Table of Symbols}
\label{app:symbols}
\begin{table}[H]
%\small
\begin{tabular}{|c|l|l|}
\hline
symbol & meaning \\
\hline
$O,N$ & center of billiard and vertex count \\
$L,J$ & perimeter and Joachimsthal's constant \\
$a,b$ & billiard major, minor semi-axes \\
$a'',b''$ & caustic major, minor semi-axes\\
$f_1,f_2$ & foci \\
$P_i,\theta_i$ &
$N$-periodic vertices and angle \\
$d_{j,i}$ & distance $|P_i-f_j|$ \\
$A$ & $N$-periodic area \\
\hline
$\rho$ & radius of the inversion circle \\
$P_{j,i}^{\dagger}$ & vertices of the inversive polygon wrt $f_j$ \\
$L_j^\dagger,A_j^\dagger$ & perimeter, area of inversive polygon wrt $f_j$\\
$\theta_{j,i}^\dagger$ & ith angle of inversive polygon wrt $f_j$ \\
\hline
\end{tabular}
\caption{Symbols used in the invariants. Note $i\in[1,N]$ and $j=1,2$.}
\label{tab:symbols}
\end{table}

\bibliographystyle{maa}
\bibliography{elliptic_billiards_v3,authors_rgk_v3} 

\begin{thebibliography}{10}
\expandafter\ifx\csname urlstyle\endcsname\relax
 \providecommand{\url}[1]{doi:\discretionary{}{}{}#1}\else
 \providecommand{\url}{doi:\discretionary{}{}{}\begingroup
  \urlstyle{rm}\Url}\fi

\bibitem{akopyan12}
Akopyan, A. (2012).
\newblock Conjugation of lines with respect to a triangle.
\newblock \emph{Journal of Classical Geometry}, 1: 23--31.

\bibitem{bialy2020-invariants}
Bialy, M., Tabachnikov, S. (2020).
\newblock {Dan Reznik's} identities and more.
\newblock \emph{Eur. J. Math.}
\newblock {doi:}10.1007/s40879-020-00428-7.

\bibitem{ferreol2020-booth}
Ferréol, R. (2020).
\newblock Booth's curve.
\newblock Mathcurve Portal.
\newblock \url{https://mathcurve.com/courbes2d.gb/booth/booth.shtml}.

\bibitem{ferreol2020-limacon}
Ferréol, R. (2020).
\newblock Pascal's limaçon.
\newblock Mathcurve Portal.
\newblock \url{https://mathcurve.com/courbes2d.gb/limacon/limacon.shtml}.

\bibitem{garcia2019-incenter}
Garcia, R. (2019).
\newblock Elliptic billiards and ellipses associated to the 3-periodic orbits.
\newblock \emph{American Mathematical Monthly}, 126(06): 491--504.

\bibitem{garcia2020-self-intersected}
Garcia, R., Reznik, D. (2020).
\newblock Invariants of self-intersected and inversive {N}-periodics in the
  elliptic billiard.
\newblock {arXiv:}2011.06640.

\bibitem{garcia2020-ellipses}
Garcia, R., Reznik, D., Koiller, J. (2020).
\newblock Loci of 3-periodics in an elliptic billiard: why so many ellipses?
\newblock {arXiv:}2001.08041.

\bibitem{garcia2020-new-properties}
Garcia, R., Reznik, D., Koiller, J. (2020).
\newblock New properties of triangular orbits in elliptic billiards.
\newblock \emph{Amer. Math. Monthly}, to appear.

\bibitem{kaloshin2018}
Kaloshin, V., Sorrentino, A. (2018).
\newblock On the integrability of {B}irkhoff billiards.
\newblock \emph{Phil. Trans. R. Soc.}, A(376).

\bibitem{kimberling1993_rocky}
Kimberling, C. (1993).
\newblock Triangle centers as functions.
\newblock \emph{Rocky Mountain J. Math.}, 23(4): 1269--1286.
\newblock \url{doi.org/10.1216/rmjm/1181072493}.

\bibitem{etc}
Kimberling, C. (2019).
\newblock Encyclopedia of triangle centers.
\newblock \url{faculty.evansville.edu/ck6/encyclopedia/ETC.html}.

\bibitem{odehnal2011-poristic}
Odehnal, B. (2011).
\newblock Poristic loci of triangle centers.
\newblock \emph{J. Geom. Graph.}, 15(1): 45--67.

\bibitem{preparata1988}
Preparata, F., Shamos, M. (1988).
\newblock \emph{Computational Geometry - An Introduction}.
\newblock Berlin: Springer-Verlag, 2nd ed.

\bibitem{reznik2021-circum}
Reznik, D., Garcia, R. (2021).
\newblock Circuminvariants of 3-periodics in the elliptic billiard.
\newblock \emph{Intl. J. Geometry}, 10(1): 31--57.

\bibitem{reznik2020-ballet}
Reznik, D., Garcia, R., Koiller, J. (2020).
\newblock The ballet of triangle centers on the elliptic billiard.
\newblock \emph{Journal for Geometry and Graphics}, 24(1): 79--101.

\bibitem{reznik2020-intelligencer}
Reznik, D., Garcia, R., Koiller, J. (2020).
\newblock Can the elliptic billiard still surprise us?
\newblock \emph{Math Intelligencer}, 42: 6--17.

\bibitem{reznik2020-invariants}
Reznik, D., Garcia, R., Koiller, J. (2021).
\newblock Fifty new invariants of {N}-periodics in the elliptic billiard.
\newblock \emph{Arnold Math. J.}
\newblock {doi}:10.1007/s40598-021-00174-y.

\bibitem{roitman2021-bicentric}
Roitman, P., Garcia, R., Reznik, D. (2021).
\newblock New invariants of {Poncelet-Jacobi} bicentric polygons.
\newblock {arXiv:}2103.11260.

\bibitem{olga14}
Romaskevich, O. (2014).
\newblock On the incenters of triangular orbits on elliptic billiards.
\newblock \emph{Enseign. Math.}, 60(3-4): 247--255.

\bibitem{schwartz2016-com}
Schwartz, R., Tabachnikov, S. (2016).
\newblock Centers of mass of {P}oncelet polygons, 200 years after.
\newblock \emph{Math. Intelligencer}, 38(2): 29--34.

\bibitem{stachel2020-private}
Stachel, H. (2020).
\newblock Expression for {Joachimsthal}'s constant.
\newblock Private Communication.

\bibitem{sergei91}
Tabachnikov, S. (2005).
\newblock \emph{Geometry and Billiards}, vol.~30 of \emph{Student Mathematical
  Library}.
\newblock Providence, RI: American Mathematical Society.
\newblock \url{bit.ly/2RV04CK}.

\bibitem{mw}
Weisstein, E. (2019).
\newblock Mathworld.
\newblock \url{mathworld.wolfram.com}.

\end{thebibliography}

\end{document}